\documentclass[10pt]{article}
\usepackage{amssymb}
\usepackage{amsmath, amsthm}
\usepackage{epsfig, latexsym}
\usepackage{latexsym}

\newcommand{\pbmod}{\textup{\,mod }}
\newcommand{\fgr}{\textup{FG}}

\title{The Identity Correspondence Problem \\ and its Applications}
\author{Paul C. Bell, Igor Potapov\thanks{
Department of Computer Science, The University of Liverpool,
Email: p.c.bell@liverpool.ac.uk (P. Bell), potapov@liverpool.ac.uk (I. Potapov)}}
\date{}

\newtheorem{theorem}{Theorem}
\newtheorem{corollary}[theorem]{Corollary}
\newtheorem{lemma}[theorem]{Lemma}
\newtheorem{prob}[theorem]{Problem}
\newtheorem{definition}[theorem]{Definition}
\begin{document}
\maketitle

\begin{abstract}
In this paper we study several closely related fundamental problems
for words and matrices. First, we introduce the Identity
Correspondence Problem (ICP): whether a finite set of pairs of words
(over a group alphabet) can generate an identity pair by a sequence
of concatenations. We prove that ICP is undecidable by a reduction
of Post's Correspondence Problem via several new encoding
techniques. In the second part of the paper we use ICP to answer a
long standing open problem concerning matrix semigroups: ``Is it
decidable for a finitely generated semigroup $S$ of integral square 
matrices whether or not the identity matrix belongs to $S$?''. We
show that the problem is undecidable starting from dimension four
even when the number of matrices in the generator is $48$. From this
fact, we can immediately derive that the fundamental problem of
whether a finite set of matrices generates a group is also
undecidable. We also answer several questions for matrices over
different number fields. Apart from the application to matrix
problems, we believe that the Identity Correspondence Problem will
also be useful in identifying new areas of undecidable problems in
abstract algebra, computational
questions in logic and combinatorics on words. \\

{\bf Keywords:} Combinatorics on Words, Group problem, Post's
Correspondence Problem, Matrix Semigroups, Undecidability.
\end{abstract}

\section{Introduction}

Combinatorics on words has strong connections to several areas of
mathematics and computing. It is well known that words are very
suitable objects to formulate fundamental properties of
computations. One such property that may be formulated in terms of
operations on words is the exceptional concept of undecidability. A
problem is called undecidable if there exists no algorithm that can
solve it.  A famous example is Post's Correspondence Problem (PCP)
originally proved undecidable by Emil Post in 1946 \cite{Po46}. It plays a central
role in computer science due to its applicability for showing the
undecidability of many computational problems in a very natural and
simple way.

There are surprisingly many easily defined problems whose
decidability status is still open.  In some cases we believe that an
algorithm solving the problem may exist, but finding it would
require the solution to fundamental open problems in mathematics.
For other problems, the current tools for showing undecidability are
not directly applicable and new techniques need to be invented to
explore the border between decidable and undecidable problems.

In this paper, we introduce the Identity Correspondence Problem (ICP)
in the spirit of Post's Correspondence Problem : whether a
finite set of pairs of words (over a group alphabet) can generate an
identity pair by a sequence of concatenations. We prove that ICP is
undecidable by a reduction of Post's Correspondence Problem via
several new encoding techniques that are used to guarantee the
existence of an identity pair only in the case of a correct solution
existing for the PCP instance.  It is our belief that the 
Identity Correspondence Problem may be useful in identifying new areas 
of undecidable problems related to computational questions in abstract 
algebra, logic and combinatorics on words.

In the second part of the paper, we use the Identity Correspondence
Problem to answer several long standing open problems concerning
matrix semigroups \cite{BCK04}. Taking products of matrices is one of the
fundamental operations in mathematics. However, many computational
problems related to the analysis of matrix products are
algorithmically hard and even undecidable. Among the oldest results
is a remarkable paper by M.~Paterson, where he shows that it is
undecidable whether the multiplicative semigroup generated by a
finite set of $3 \times 3$ integer matrices contains the zero matrix
(also known as the mortality problem), see \cite{Pa70}. Since then,
many results were obtained about checking the freeness, boundedness
and finiteness of matrix semigroups and the decidability of
different reachability questions such as the membership problem,
vector reachability, scalar reachability etc. See
\cite{BP05,BP08,BP08-SOF,BP08quat,BT00,CHK99,CK05,HHH06} 
for several related decidability results.

The membership problem asks whether a particular matrix is contained
within a given semigroup. The membership problem is undecidable for
$3 \times 3$ integral matrix semigroups due to Paterson's results
and also for the special linear group SL($4$, $\mathbb{Z}$) of $4 \times 4$ integer
matrices of determinant $1$, shown by Mikhailova \cite{Mi58}.

Another important problem in matrix semigroups is the Identity
Problem: Decide whether a finitely generated integral matrix semigroup
contains the identity matrix.  The Identity Problem is equivalent to
the following Group Problem: given a finitely generated semigroup
$S$, decide whether a subset of the generator of $S$ generates a non-trivial group. 
In general, it is undecidable whether or not the monoid described by a given
finite representation is a group. However, this decision problem is
reducible to a very restricted form of the uniform word problem and
it does not immediately imply that the Group Problem in finitely generated 
semigroups (without a set of relations) is undecidable \cite{O86}. 

The question about the membership of the identity matrix for matrix
semigroups is a well known open problem and was recently stated in
``Unsolved Problems in Mathematical Systems and Control Theory'',
\cite{BCK04} and also as Problem~5 in \cite{Har09}. The embedding 
methods used to show undecidability in other results do not appear 
to work here \cite{BCK04}. As far as we know, only two decidability 
results are known for the Identity Problem.  Very recently the first 
general decidability result for this problem was proved in the case 
of  $2 \times 2$ integral matrix semigroups, see \cite{CK05}. It is 
also known that in the special case of commutative matrix semigroups, 
the problem is decidable in any dimension \cite{BB96}.

In this paper we apply ICP to answer the long standing open problem:
``Is it decidable for a finitely generated semigroup $S$ of square
integral matrices whether or not the identity matrix belongs to
$S$?''. We show that the Identity Problem is undecidable starting
from dimension four even when the number of matrices in the
generator is fixed. In other words, we can define a class of finite
sets $\{M_1,M_2,\ldots ,M_k \}$ of four dimensional matrices such
that there is no algorithm to determine whether or not the identity
matrix can be represented as a product of these matrices. From this
fact, we can immediately derive that the fundamental problem of
whether a finite set of $4 \times 4$ matrices generates a group is
also undecidable. In our proofs we use the fact that free groups can
be embedded into the multiplicative group of $2 \times 2$ integral
matrices. This allows us to transfer the undecidability of ICP into
undecidability results on matrices.

We also provide a number of other corollaries. In particular, the
Identity and Group problems are undecidable for double quaternions
and a set of rotations on the $3$-sphere. Therefore, there is no
algorithm to check whether a set of linear transformations or a set
of rotations in dimension $4$ is reversible. Also, the question of
whether any diagonal matrix can be generated by a $4 \times 4$
integral matrix semigroup is undecidable.

\section{Identity Correspondence Problem}

{\bf Notation:} Given an alphabet $\Sigma = \{a,b\}$, we denote the
concatenation of two letters $x, y \in \Sigma$ by $xy$ or $x \cdot
y$. A \emph{word} over $\Sigma$ is a concatenation of letters from
alphabet $\Sigma$, i.e., $w = w_1w_2 \cdots w_k \in \Sigma^*$. We
denote throughout the paper the \emph{empty word} (or identify element)
by $\varepsilon$. We shall denote a \emph{pair of words} by either 
$(w_1, w_2)$ or $\frac{w_1}{w_2}$.

The free group over a generating set $H$ is denoted by $\fgr(H)$,
i.e., the free group over two elements $a$ and $b$ is denoted as
$\fgr(\{a,b\})$. 
For example, the elements of $\fgr(\{a,b\})$ are all the words
over the alphabet $\{a,b,a^{-1},b^{-1}\}$ that are reduced, i.e.,
that contain no subword of the form $x \cdot x^{-1}$ or $x^{-1}
\cdot x$ (for $x \in \{a,b\}$). Note that $x \cdot x^{-1} = x^{-1} \cdot x = \varepsilon$.



\begin{prob}\label{IDPCP_prob} Identity Correspondence Problem (ICP) -
Let $\Sigma = \{a,b\}$ be a binary alphabet and 
$$\Pi = \{(s_1, t_1), (s_2, t_2), \ldots, (s_m, t_m)\} \subseteq \fgr(\Sigma) \times \fgr(\Sigma).
$$ 
Determine if there exists a nonempty finite sequence of
indices $l_1, l_2, \ldots, l_k$ where $1 \leq l_i \leq m$ such that 
$$
s_{l_1}s_{l_2}\cdots s_{l_k} = t_{l_1}t_{l_2} \cdots t_{l_k} = \varepsilon,
$$
where $\varepsilon$ is the empty word (identity).
\end{prob}

A first step towards the proof of undecidability of
Problem~\ref{IDPCP_prob} was shown in \cite{BP05} where the
following theorem was presented (although in a different form).


\begin{theorem}\label{ICPCP}\cite{BP05} - Index Coding PCP -
Let $\Sigma = \{a,b\}$ be a binary alphabet and
$$X = \{(s_1, t_1),(s_2, t_2), \ldots, (s_f, t_f)\} \subseteq \fgr(\Sigma) \times \fgr(\Sigma).$$
It is undecidable to determine if there exists a finite sequence
$l_1, l_2, \ldots, l_k$ where $1 \leq l_i \leq f$ and \emph{exactly
one} $l_i = f$ such that
$$
s_{l_1} s_{l_2} \cdots s_{l_k} = t_{l_1} t_{l_2} \cdots t_{l_k} =
\varepsilon.
$$
\end{theorem}

Unfortunately, Theorem~\ref{ICPCP} cannot be directly used to prove the Identity
Problem or the Group Problem are undecidable. We may, however, immediately use
Problem~\ref{IDPCP_prob} for this purpose (and do so in Section~\ref{app_sec}) once we 
have proved that it is undecidable.

The reason Theorem~\ref{ICPCP} does not prove
Problem~\ref{IDPCP_prob} is undecidable is the restriction that the
final pair of words $(s_f, t_f)$ is used exactly one time. Despite
many attempts, it is not clear how one may remove this restriction
in the construction of the proof, since it is essential that this
pair be used once to avoid the pathological case of several
incorrect solutions cancelling with each other and producing an identity element.

The main idea of this paper is to show a new non-trivial encoding
which contains the encoding used in Theorem~\ref{ICPCP} but avoids
the requirement that a specific element be used one time. The idea
is that by encoding the set $X$ \emph{four times} using four
different alphabets and adding `borders' to each pair of words such
that for cancellation to occur, each of these alphabets must be used
in a specific (cyclic) order, any incorrect solutions using a single alphabet
will not be able to be cancelled later on. More details of this encoding with four alphabets
will be given later, in Lemmas~\ref{oneCycle}, \ref{lem_concat} and \ref{arbProduct} and the example 
that follows them, which provides some intuition as to why three alphabets is not sufficient in the encoding.

We shall reduce a restricted form of Post's Correspondence Problem
(PCP) \cite{HHH06} to the Identity Correspondence Problem in a
constructive way. We shall require the following theorem:

\begin{theorem}\label{GPCP} \cite{HHH06,Mat96} Restricted PCP - Let $\Sigma = \{a,b\}$ be a binary alphabet and
$$P = \{(u_1, v_1),(u_2, v_2), \ldots, (u_n, v_n)\} \subseteq \Sigma^* \times \Sigma^*$$
be a set of pairs of words where $n \geq 3$. It is undecidable to determine if there
exists a finite sequence of indices $l_1, l_2, \ldots, l_k$ with each $2 \leq l_i \leq n-1$ such that:
$$
u_1 u_{l_1} u_{l_2} \cdots u_{l_k} u_n = v_1 v_{l_1} v_{l_2} \cdots
v_{l_k} v_n.
$$
This result holds even for $n = 7$.
\end{theorem}

We now show the reduction of an instance of the Restricted Post's
Correspondence Problem of Theorem~\ref{GPCP} to an instance of the
Identity Correspondence Problem. Let here and throughout $\Sigma =
\{a,b\}$ and define new alphabets $\Gamma_i = \{a_i,b_i\}$ for
$1 \leq i \leq 4$ and $\Gamma_B = \{x_j| 1 \leq j \leq 8 \}$
such that the alphabets are distinct (specifically, the intersection
of the free groups generated by any two different alphabets equals
$\{\varepsilon\}$). Let us define mappings ${\delta}_i :
\fgr(\Sigma) \to \fgr(\Gamma_i)$ by ${\delta}_i (a) = a_i$,
${\delta}_i (b) = b_i$, ${\delta}_i (a^{-1}) = a^{-1}_i$ and
${\delta}_i (b^{-1}) = b^{-1}_i$ for $1 \leq i \leq 4$. Note
that each $\delta_i$ is a homomorphism that may be applied to words
over $\fgr(\Sigma)$ in a natural way.

Let $\Gamma = \Gamma_1 \cup \Gamma_2 \cup \Gamma_3 \cup \Gamma_4
\cup \Gamma_B$. Define $\phi_i: \mathbb{Z}^+ \to \{a_i,b_i\}^*$ by
$\phi_i(j) = a_i^jb_i$. Similarly, let $\psi_i: \mathbb{Z}^+ \to
\{a_i^{-1},b_i^{-1}\}^*$ be defined by $\psi_i(j) =
(a_i^{-1})^jb_i^{-1}$. These morphisms will be used to ensure a
product is in a specific order. As an example of these morphisms we see that
$\phi_2(3) = a_2a_2a_2b_2$ and $\psi_3(2) = a_3^{-1}a_3^{-1}b_3^{-1}$.

Let $P = \{(u_1, v_1),(u_2, v_2), \ldots, (u_n, v_n)\} \subseteq
\Sigma^* \times \Sigma^*$ be a given Restricted PCP instance. We shall define an
instance of ICP consisting of a set of 
$8(n-1)$ pairs of words:
$$W = W_0 \cup W_1 \cup \ldots  \cup W_{15} \subseteq \fgr(\Gamma) \times \fgr(\Gamma)$$
$$
\begin{array}{rll}  &
W_0 = \Big\{\frac{x_8}{x_8}\cdot \frac{v_{11}^{-1}u_{11}}{b_1}\cdot
\frac{x_1^{-1}}{x_1^{-1}}\Big\}, & W_1 = \Big\{\frac{x_1}{x_1}\cdot
\frac{u_{1j}}{\phi_1(j)}\cdot \frac{x_1^{-1}}{x_1^{-1}}| 2 \leq j
\leq n-1 \Big\}, \\ & W_2 = \Big\{\frac{x_1}{x_1}\cdot
\frac{u_{1n}v_{1n}^{-1}}{b_1^{-1}}\cdot\frac{x_2^{-1}}{x_2^{-1}}\Big\},
& W_3 = \Big\{\frac{x_2}{x_2}\cdot
\frac{v_{1j}^{-1}}{\psi_1(j)}\cdot\frac{x_2^{-1}}{x_2^{-1}} | 2 \leq
j \leq n-1 \Big\}, \\ & W_4 = \Big\{\frac{x_2}{x_2}\cdot
\frac{v_{21}^{-1}u_{21}}{b_2}\cdot\frac{x_3^{-1}}{x_3^{-1}}\Big\}, &
W_5 = \Big\{\frac{x_3}{x_3}\cdot
\frac{u_{2j}}{\phi_2(j)}\cdot\frac{x_3^{-1}}{x_3^{-1}}| 2 \leq j
\leq n-1 \Big\}, \\ & W_6 = \Big\{\frac{x_3}{x_3}\cdot
\frac{u_{2n}v_{2n}^{-1}}{b_2^{-1}}\cdot\frac{x_4^{-1}}{x_4^{-1}}\Big\},
& W_7 = \Big\{\frac{x_4}{x_4}\cdot
\frac{v_{2j}^{-1}}{\psi_2(j)}\cdot\frac{x_4^{-1}}{x_4^{-1}} | 2 \leq
j \leq n-1 \Big\}, \\ & W_8 = \Big\{\frac{x_4}{x_4}\cdot
\frac{v_{31}^{-1}u_{31}}{b_3}\cdot\frac{x_5^{-1}}{x_5^{-1}}\Big\}, &
W_9 = \Big\{\frac{x_5}{x_5}\cdot
\frac{u_{3j}}{\phi_3(j)}\cdot\frac{x_5^{-1}}{x_5^{-1}}| 2 \leq j
\leq n-1 \Big\},\\ & W_{10} = \Big\{\frac{x_5}{x_5}\cdot
\frac{u_{3n}v_{3n}^{-1}}{b_3^{-1}}\cdot\frac{x_6^{-1}}{x_6^{-1}}\Big\},
& W_{11} = \Big\{\frac{x_6}{x_6}\cdot
\frac{v_{3j}^{-1}}{\psi_3(j)}\cdot\frac{x_6^{-1}}{x_6^{-1}} | 2 \leq
j \leq n-1 \Big\}, \\ & W_{12} = \Big\{\frac{x_6}{x_6}\cdot
\frac{v_{41}^{-1}u_{41}}{b_4}\cdot\frac{x_7^{-1}}{x_7^{-1}}\Big\}, &
W_{13} = \Big\{\frac{x_7}{x_7}\cdot
\frac{u_{4j}}{\phi_4(j)}\cdot\frac{x_7^{-1}}{x_7^{-1}}| 2 \leq j
\leq n-1 \Big\}, \\ & W_{14} = \Big\{\frac{x_7}{x_7}\cdot
\frac{u_{4n}v_{4n}^{-1}}{b_4^{-1}}\cdot\frac{x_8^{-1}}{x_8^{-1}}\Big\},
& W_{15} = \Big\{\frac{x_8}{x_8}\cdot
\frac{v_{4j}^{-1}}{\psi_4(j)}\cdot\frac{x_8^{-1}}{x_8^{-1}}| 2 \leq
j \leq n-1 \Big\},
\end{array}
$$
where $u_{ik} = {\delta}_i (u_k)$, $v_{ik} = {\delta}_i (v_k)$ for
$1 \leq k \leq n$ and $1 \leq i \leq 4$, thus $u_{ik} \in \{a_i, b_i\}^*$ and 
$v_{ik}^{-1} \in \{a_i^{-1}, b_i^{-1}\}^*$.
Given any two words $w_1,w_2 \in \fgr(\Gamma)$, recall that we
denote by $\frac{w_1}{w_2}$ the pair of words $(w_1,w_2) \in
\fgr(\Gamma) \times \fgr(\Gamma)$ in the above table.

\begin{figure}[h]
\begin{center}
\includegraphics[scale=0.50]{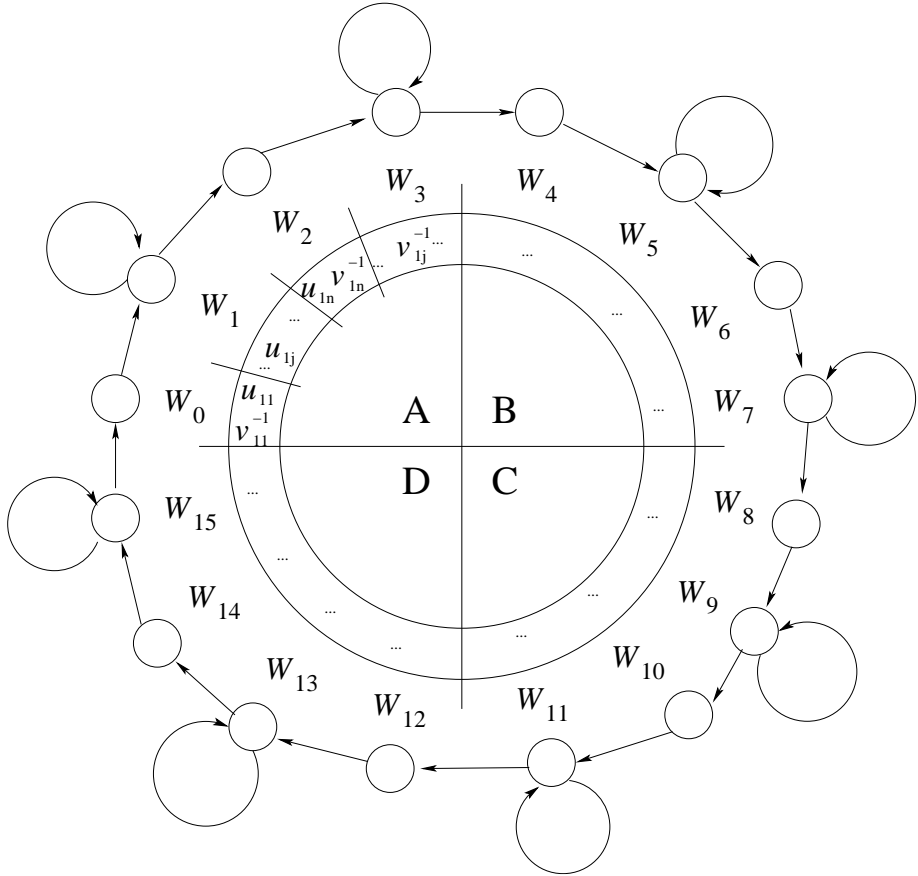}
\caption{The structure of a product which forms the identity.}\label{fig_1}
\end{center}
\end{figure}

Note that each word in each pair from $W_i$ has a so called
`border letter' on the left and right from 
$\fgr(\Gamma_B)$. These are used to restrict the type of sequence\footnote{ 
The only sequences that may lead to an identity pair should
be of the form of a cycle or a nested insertion of cycles as we shall show in Lemma~\ref{lem_cycles}.}
that can  lead to an identity pair.
The central element of each word (i.e. excluding the
`border letters') corresponds to particular words from $P$ and we
encode instance $P$ four times separately, first in
$W_0,W_1,W_2,W_3$, secondly in $W_4, W_5, W_6, W_7$ etc. using
different alphabets for each encoding \footnote{In the case of an incorrect solution for 
the Restricted PCP instance (i.e. an index sequence $i_1, \ldots, i_k$ such that 
$u_{i_1} \cdots u_{i_k} \neq v_{i_1} \cdots v_{i_k}$), the use of different alphabets for the four parts creates a sequence 
of non-empty parts that cannot be trivially cancelled from the left or right side.}. 
This may be seen in Figure~\ref{fig_1}, where $A,B,C$ and $D$ each separately encode
instance $P$. 

This forms the set $W = \{(s_1, t_1), (s_2, t_2), \ldots, (s_m, t_m)\} \subset 
\fgr(\Gamma) \times \fgr(\Gamma)$. Let us define the $s_i$-word to mean the first word
from pair of words $(s_i,t_i)$ and the $t_i$-word for the second word of this pair.
The $t_i$-words from each pair in $W$ use an encoding which ensures that the
set of $s_i$-words is concatenated in a particular order within
each $A,B,C$ and $D$ part. We show in Lemma~\ref{SecondWordsEncoding} that this encoding enforces a
correct encoding of the Restricted PCP instance $P$ within each part if that part
gets reduced to two letters in the second word (the first and
last `border letters'). We adapt here our recently 
introduced index encoding technique from \cite{BP05}.

One of the important encoding concepts is a \emph{cycle} of set $W$. 
We see that the first and last letters from each word of any pair of words 
from set $W_i \subset W$ only cancel with a pair of words from set 
$W_{i+1 \pbmod 16}$ for $0 \leq i \leq 15$ and with elements from $W_i$ 
itself if $i \pbmod 2 \equiv 1$. We shall now define a `cycle' of set $W$.
\begin{definition}
An element $w \in W^*$ is called a cycle of $W$ if it is of the form:
\begin{eqnarray}
w = w_{i} \cdot w_{(i+1) \pbmod 16} \cdot \ldots \cdot w_{(i+15)
\pbmod 16} \in W^* \label{wForm1}
\end{eqnarray}
for some $i$: $0 \leq i \leq 15$, where $w_y \in W_y$ if $y \pbmod 2
\equiv 0$ and $w_y \in W_y^*$ if $y \pbmod 2 \equiv 1$.
\end{definition}

For example a cycle could use element $W_4$ followed by a product of
elements from $W_5$, then element $W_6$, followed by a product of
elements from $W_7$ etc. As previously mentioned, the idea of the
encoding is that a correct solution to the Restricted PCP instance $P$ will be
encoded \emph{four times} in a correct solution to $W$, in elements
from $\{W_0, \ldots, W_3\}, \{W_4, \ldots, W_7\}, \{W_8, \ldots,
W_{11}\}$ and $\{W_{12}, \ldots, W_{15}\}$ separately.

We now define a \emph{pattern generated by cycle insertions}. By this, we mean a product where cycles can be 
inserted within other cycles or appended to the end of them. For example, given an
element $q_1q_2q_3q_4q_3'q_1'q_5 \in W^*$ where $q_1q_1'$, $q_2$, $q_3q_3'$, $q_4$ and $q_5$ 
are all cycles, then this would form a pattern generated by cycle insertions since 
it can be decomposed into cycles being nested or concatenated in the required way.

\begin{lemma}\label{lem_cycles}
If instance $W$ of the Identity Correspondence Problem has a solution, it must be 
constructed by an element $w \in W^*$ which forms either a single {\sl cycle} or a 
pattern generated by cycle insertions (including concatenation).
\end{lemma}
\begin{proof} 
It is not difficult to see that the border symbols from $\Gamma_B$ give us constraints on the 
type of patterns which can be considered as possible solutions to the ICP instance, 
i.e., which may have some form of word cancellation. These constraints can be considered as the 
state system represented in Figure~\ref{fig_1} and require that a cycle is completed in a clockwise 
direction. Note that for any $w \in W_y^*$ where $y \pbmod 2 \equiv 1$, it holds that $w$ is not 
equal to $(\varepsilon,\varepsilon)$ since the left and right borders of each word are separated by 
a nonempty word from an alphabet not containing inverse elements. Let us assume that we have a pair 
of words from $W_i$ for some $i$. The only possible way to cancel its border symbols is to complete 
a chain of cancellations by inverse border elements which will correspond to a clockwise traversal 
of states represented in Figure~\ref{fig_1}. Since at any time we can start to build a new cycle and 
all cycles must be completed we have that the only sequence of word pairs that can be equal to 
$(\varepsilon,\varepsilon)$ must be represented as a single cycle or a pattern generated by cycle 
insertions.
\end{proof}

\begin{definition}
For any product $Y \in W^*$ we shall denote by a \emph{decomposition
by parts} of $Y$, the decomposition $Y = Y_1 Y_2 \cdots Y_k$ where
for each $1 \leq i \leq k$, if $Y_i \subset \fgr(\Gamma_i \cup
\Gamma_B) \times \fgr(\Gamma_i \cup \Gamma_B)$ then $Y_{i+1}
\subset \fgr(\Gamma_j \cup \Gamma_B) \times \fgr(\Gamma_j \cup
\Gamma_B)$ where $1 \leq i,j \leq 4$ and $i\neq j$.
\end{definition}
For a cycle $Q$, the decomposition by parts of $Q$ clearly gives
either $4$ or $5$ parts in the decomposition. For example, we may
have $Q = X_1 X_2 X_3 X_4 X_5$ where $X_1 \in \fgr(\Gamma_i \cup
\Gamma_B) \times \fgr(\Gamma_i \cup \Gamma_B)$ and thus $X_2 \in
\fgr(\Gamma_{(i+1 \pbmod 4)} \cup \Gamma_B) \times
\fgr(\Gamma_{(i+1 \pbmod 4)} \cup \Gamma_B)$ etc. $X_5$ is
either empty or uses the same alphabet as $X_1$.


The $s_i$-words from each $A,B,C,D$ part of
Figure~\ref{fig_1} will store all words from the instance of
Restricted PCP, $P$ separately using distinct alphabets. If we concatenate the $s_i$-words 
of one of these parts in the correct order and have the empty
word (excluding initial and final `border letters'), then this
corresponds to a solution of $P$. By a correct order, we mean that
if we have $u_{i1}u_{i2} \cdots u_{ik}$ for example, then they
should be concatenated with $(v_{i1}v_{i2} \cdots v_{ik})^{-1} =
v_{ik}^{-1} \cdots v_{i2}^{-1} v_{i1}^{-1}$. If the concatenation of
these words equals $\varepsilon$, then we have a correct solution to
$P$.

Let us here illustrate this fact with a simple example. Take a (standard)
PCP instance $P= \{(aab, a),(a,baa)\}$. Clearly we have a solution to this instance since
$(aab,a)(aab,a)(a,baa)(a,baa) = (aabaabaa, aabaabaa)$.
Using the above encoding, we can alternatively write this as:
$$
(aab)(aab)(a)(a) \,\cdot\, (baa)^{-1}(baa)^{-1}(a)^{-1}(a)^{-1} = \varepsilon,
$$
which can also be seen as a solution where the words on the left are from the first words of 
each pair in $P$ and the words on the right are the inverse of the second words from $P$. 
The idea is that on the right, using the inverse elements of the alphabet, we should have 
a palindrome of the word on the left and they should occur in the correct order. Here we 
used the sequence $1,1,2,2$ on the left thus used the reverse sequence on the right, namely $2,2,1,1$.

The encoding in the second words using $\phi_i, \psi_i$ and $\{b_i,
b_i^{-1} | 1 \leq i \leq 4\}$ is used to ensure that any solution to
$W$ \emph{must} use such a correct ordering in each $A,B,C,D$
part. The next lemma formalizes this concept and is a  modification of 
the technique presented in \cite{BP05}. It also can be seen as a variant of 
Index Coding PCP, see \cite{BP08-SOF}, which is simpler to prove.

\begin{lemma}\label{SecondWordsEncoding}
Given any part $X \in \fgr(\Gamma_j \cup
\Gamma_B) \times \fgr(\Gamma_j \cup \Gamma_B)$, if the second
word of $X$ consists of only the initial and final `border letters' 
$x_p x^{-1}_q \in \Gamma_B^*$, then the second word of $X$ must be of the form
$$ x_p \cdot b_j \phi_j(z_1) \phi_j(z_2) \cdots \phi_j(z_k) \cdot b_j^{-1} \cdot
\psi_j(z_k) \cdots \psi_j(z_2) \psi_j(z_1)   \cdot x^{-1}_q, 
$$
where $2 \leq z_1, z_2, \ldots, z_k \leq n-1$. 
(This corresponds to a `correct' palindromic encoding of the Restricted PCP instance $P$ within this part.
We see that all elements except $x_p$ and $x_q^{-1}$ will be cancelled.)
\end{lemma}
\begin{proof}
Since $X$ is a single part, we see that $(p,q) \in \{(8,2), (2,4), (4,6), (6,8) \}$ depending
on the type of part $X$.
Let us consider the case that $X$ is a product over elements from $W_0
\cup W_1 \cup W_2 \cup W_3$ (part A in Figure~\ref{fig_1}). The proof
for the other `parts', $B,C$ and $D$ is analogous. Consider the
morphisms used in the second words of these elements. If we have for
example a word starting with the element from $W_0$, by the choice
of `border letters', it must be followed by an element from $W_1^*$ or
$W_2$ for cancellation to occur. In the former case (excluding `border
letters') it will thus be of the form $b_1 a_1^{z_1}b_1 \cdot a_1^{z_2}b_1 \cdots
a_1^{z_k}b_1$ where each $2 \leq z_i \leq n-1$. 
The only way to cancel this final $b_1$ is to eventually
use $W_2$ (even if we use no element from $W_1^*$) since this is
the only element whose second word starts with $b_1^{-1}$ and this is the only
other element within $W$ cancelling the `border letter' of $W_1$.

After this we must use an element from $W_3^*$ to cancel the $a_1$
values since each $\psi(i)$ starts with $a_i^{-1}$. It is not
difficult to see that we in fact must use these elements in the
order $\psi(z_k) \cdot \psi(z_{k-1}) \cdots \psi(z_1)$
otherwise the `$b_1^{-1}$' at the end of some $\psi(z_j)$ will
not be cancelled on the left. The only way to cancel this
`$b_1^{-1}$' would be to use $W_1$ but this cannot follow $W_3$ by
the choice of `border letters'. With a correct sequence of $W_3$
elements, all the letters of the second words will cancel leaving
the empty word $\varepsilon$ (again excluding the `border letters').
If we do not use this sequence, by the choice of the morphisms
$\phi$ and $\psi$, the letters cannot be reduced to $\varepsilon$.
See \cite{BP05} for further details.

Finally note that if we do not start with the element from $W_0$
then, since the left `border letter' of the pair of words in this
element is $x_8$, we cannot use it to cancel the product later on,
since this border essentially splits the pair of words in two. It is
not difficult to see that without this element we cannot reduce a
product to $\varepsilon$ however since without the $b_i$ element in
the second word to cancel the last letter of $W_2$ or $W_3$
elements, they cannot be reduced. Thus we must have the given
structure given in the lemma. 
\end{proof}


\begin{lemma}\label{PCPSol}
If there exists a solution to the Restricted PCP instance $P$, then there
exists a solution to the Identity Correspondence Problem instance
$W$.
\end{lemma}
\begin{proof}
Assume we have a solution to $P$ with indices $2 \leq i_1, i_2,
\ldots, i_k \leq n-1$, i.e., $u_1u_{i_1} \cdots u_{i_k}u_n =
v_1v_{i_1} \cdots v_{i_k}v_n$.
We can explicitly define a product which will give a correct solution to the 
ICP instance $W$. Define a word $w = w_0w_1 \cdots w_{15} \in W^*$ such that each 
$w_i \in W_i^*$. If $i \pbmod 1 \equiv 0$, then $|w_i| = 1$. If $i \pbmod 4 \equiv 1$, then
$w_i$ will be chosen from $W_i^*$ using the indices $2 \leq i_1, i_2, \ldots, i_k \leq n-1$ for $j$.
Finally, if $i \pbmod 4 \equiv 3$, then $w_i$ will be chosen from $W_i^*$ using the indices 
$2 \leq i_k, i_{k-1}, \ldots, i_1 \leq n-1$ for $j$. A simple computation shows that since this sequence
gave a correct solution to $P$, then $w$ will be equal to $(\varepsilon, \varepsilon)$ and thus a solution
to the ICP instance $W$.
\end{proof}


Let us introduce several notations which will be useful for the analysis of 
cancellations that may occur in the construction. We shall define four 
`types' of parts, $A, B, C, D$ where type $A$ parts use alphabet 
$\fgr(\Gamma_1 \cup \Gamma_B) \times \fgr(\Gamma_1 \cup \Gamma_B)$, 
type $B$ parts use $\fgr(\Gamma_2 \cup \Gamma_B) \times \fgr(\Gamma_2 \cup \Gamma_B)$, 
type $C$ parts use $\fgr(\Gamma_3 \cup \Gamma_B) \times \fgr(\Gamma_3 \cup \Gamma_B)$ 
and type $D$ parts use $\fgr(\Gamma_4 \cup \Gamma_B) \times \fgr(\Gamma_4 \cup \Gamma_B)$ 
as in Figure~\ref{fig_1}. A cycle thus has a decomposition which is a permutation of $ABCD$.

We shall now define a function $\zeta: W^* \to \mathbb{N}$. Given any product 
$Y \in W^*$ with the decomposition by parts $Y = Y_1 Y_2 \cdots Y_k$, we first define the 
set of pairs of words $\{Z_1, Z_2, \ldots, Z_k\}$ where $Z_i$ is a
pair of words constructed from $Y_i$ where we exclude the initial and
final letters (from $\Gamma_B$) in each pair of words in $Y_i$. 
We let $\zeta(Y)$ denote the sum 
of non-identity words from $\{Z_1, Z_2, \ldots, Z_k\}$. Note that 
$Z_i \in \fgr(\Gamma_j \cup \Gamma_B) \times \fgr(\Gamma_j \cup \Gamma_B)$ 
for some $1 \leq j \leq 4$.

Thus for a single cycle $Q$, $0 \leq \zeta(Q) \leq 10$ since it can be
decomposed to a maximum of $5$ parts. If the first and second words
in each decomposed part have a non-reducible word in between the
borders, we have that $\zeta(Q)$ is equal to $10$. If $\zeta(Q)$
equals $0$, it means that all words in between the border elements
are reducible to identity.

\begin{lemma}\label{oneCycle}
If there exists no solution to the encoded Restricted PCP instance $P$ then for
any cycle $Q \in W^+$ having decomposition by parts  $Q = X_1 X_2 X_3 X_4 X_5$,
the following holds:
\begin{itemize}
\item $X_r \neq (\varepsilon,\varepsilon)$  for all $r$ where $1 \leq r \leq 5$;
\item $4 \leq \zeta(Q)$;
\item $Q \neq (\varepsilon,
\varepsilon)$, i.e., a single cycle cannot be a solution to the
Identity Correspondence Problem.
\end{itemize}
\end{lemma}
\begin{proof}
Let $Q$ be a single cycle of the form~(\ref{wForm1}). Since it is a
cycle, the `border letters' of each pair will all cancel with each
other and thus we may ignore letters from $\fgr(\Gamma_B)$ 
(except for the first and last such border letters).
Let $Q = X_1 X_2 X_3 X_4 X_5$ be its decomposition by parts (thus
$X_5$ can be empty and four of the `parts' use different alphabets).

Let us consider some $X_r$ where $1 \leq r \leq 5$. We will show that $X_r$ 
cannot be equal to $(\varepsilon,\varepsilon)$.

Since $Q$ is a cycle, which has a specific structure, the first word of $X_r$, 
when concatenated, equals $v_{p1}^{-1} u_{p1} u_{pj_1} \cdots u_{pj_h} u_{pn} v_{pn}^{-1}
v_{pk_l}^{-1} \cdots v_{pk_1}^{-1}$ for some $1 \leq p \leq 4$ and
$h,l \geq 0$. If $j_i = k_i$ for all $1 \leq i \leq h$ with $h = l$ 
then this is a correct encoding of the Restricted PCP instance $P$ which we
have assumed has no solution, thus this word does not equal
$\varepsilon$ in this case. Therefore the elements must not be in a
correct sequence if the first word equals $\varepsilon$. In this
case however, the second word will now not equal $\varepsilon$ by
the choice of the morphisms $\phi_i$ and $\psi_i$ as shown in
Lemma~\ref{SecondWordsEncoding}. If we have such an incorrect
ordering then when we multiply the second set of words (since also
each morphism uses a different alphabet) they never equal
$\varepsilon$ which is not difficult to see. 

So assuming that there is no solution to the Restricted PCP instance $P$, for any
part $X_r$, $X_r \neq (\varepsilon, \varepsilon)$, i.e., at least one word in the pairs of
words of each part does not equal $\varepsilon$ (even ignoring initial and final border letters). 
Thus, crucially, if there exists no solution to the encoded Restricted PCP instance $P$, then 
$4 \leq \zeta(ABCD) \leq 8$ for a cycle $ABCD \in W^*$.
\end{proof}


It follows from Lemma~\ref{oneCycle} that the statements of Lemma~\ref{lem_cycles} can be restricted further. 
Lemma~\ref{lem_cycles} asserts that the solution of ICP can be either a single cycle or a pattern that is formed 
by a nested insertion of cycles (including concatenation).
It follows from Lemma~\ref{oneCycle} that if the Restricted PCP instance $P$ does not have a solution, then a 
single cycle cannot be equal to $(\varepsilon,\varepsilon)$. We prove now that any solution to the corresponding 
ICP instance $W$ cannot be in the form of cycle insertion unless the solution is in the form of a concatenation of 
several cycles each of which starts with the same element. 

\begin{lemma}\label{lem_concat}
If there exists no solution to the Restricted PCP instance $P$, any solution to the corresponding ICP instance  
$W$ cannot be in the form of cycle insertion unless the solution is in the form of a concatenation of several cycles  
each of which starts with the same element.
\end{lemma}
\begin{proof} 
Let us assume that a sequence of indices gives us a solution to ICP in the
form $LQR$, where $L,Q, R \in W^{+}$  and $Q$ is a cycle.
We show that if $LQR=(\varepsilon,\varepsilon)$
then $Q$ is not inserted inside of any other cycles and $LQR$ is a concatenation of cycles
each of which starts with the same element.

If $LQR$ is equal to $(\varepsilon,\varepsilon)$ then $QRL=(\varepsilon,\varepsilon)$.
By $l,r,q$ let us define pairs of words constructed from $L,Q,R$ where we exclude the initial and
final border letters.

Let us assume that the single cycle $Q$ is in the form where it starts and finishes with border letters
$\frac{x_i}{x_i}$ and $\frac{x_i^{-1}}{x_i^{-1}}$,
i.e., $Q=\frac{x_i}{x_i} \cdot q \cdot \frac{x_i^{-1}}{x_i^{-1}}$ where element $q$, when reduced (i.e. removing 
consecutive inverse elements), is in $\fgr(\Gamma ') \times \fgr(\Gamma ')$ where $\Gamma '=\Gamma \setminus \Gamma_B$,
$Q \neq (\varepsilon,\varepsilon)$ and
$LR=\frac{x_k}{x_k} \cdot l \cdot \frac{x_j^{-1}}{x_j^{-1}} \cdot
\frac{x_j}{x_j}   \cdot r \cdot \frac{x_k^{-1}}{x_k^{-1}} $ for some border letters $x_j, x_k \in \Gamma_B$.

Since $q$ cannot be equal to $(\varepsilon,\varepsilon)$ by Lemma~\ref{oneCycle} and $QRL=(\varepsilon,\varepsilon)$
we have that the cycle $Q$ can only be cancelled by a concatenation with $RL$.
Thus the reduced form of $rl$ is in $\fgr(\Gamma ') \times \fgr(\Gamma ')$ and
$RL$ must therefore be in the form of concatenations
of cycles starting with a border symbol $x_i$:
$RL=\frac{x_i}{x_i} \cdot r \cdot \frac{x_k^{-1}}{x_k^{-1}} \cdot \frac{x_k}{x_k}
\cdot l \cdot \frac{x_i^{-1}}{x_i^{-1}}$. We see that $QRL$ is therefore a concatenation of cycles.

Since the cycle $Q$ can be factorized into two parts $Q_1, Q_2$ separated by border letters $x_k,x_k^{-1}$,
i.e. $Q=Q_1 Q_2 = \frac{x_i}{x_i} \cdot q_1 \cdot \frac{x_k^{-1}}{x_k^{-1}} \cdot
\frac{x_k}{x_k}  \cdot q_2 \cdot \frac{x_i^{-1}}{x_i^{-1}}$,
we have that 
$$\begin{array}{rl}LQR & = \frac{x_k}{x_k} \cdot l \cdot \frac{x_i^{-1}}{x_i^{-1}} \cdot
\frac{x_i}{x_i} \cdot q \cdot
\frac{x_i^{-1}}{x_i^{-1}} \cdot \frac{x_i}{x_i} \cdot r \cdot
 \frac{x_k^{-1}}{x_k^{-1}} \\ &
= \frac{x_k}{x_k} \cdot l \cdot \frac{x_i^{-1}}{x_i^{-1}} \cdot \frac{x_i}{x_i} \cdot q_1
\cdot  \frac{x_k^{-1}}{x_k^{-1}} \cdot   \frac{x_k}{x_k}  \cdot q_2 \cdot
\frac{x_i^{-1}}{x_i^{-1}} \cdot x_i \cdot r \cdot  \frac{x_k^{-1}}{x_k^{-1}} \\ &
=  \frac{x_k}{x_k} \cdot l \cdot  q_1  \cdot \frac{x_k^{-1}}{x_k^{-1}} \cdot
\frac{x_k}{x_k}  \cdot q_2 \cdot r \cdot \frac{x_k^{-1}}{x_k^{-1}}.
\end{array}$$

Thus $LQR$ is in the form of concatenation of cycles starting from a border letter $x_k$ as required.
\end{proof}

In the next lemma, we show that if the encoded Restricted PCP instance $P$ has no solution, 
then a concatenation of cycles also cannot form a solution. 

\begin{lemma}\label{arbProduct}
Given an instance of the Identity Correspondence Problem $W$
encoding an instance $P$ of Restricted Post's Correspondence
Problem, if there exists no solution to $P$ then for any product $X
\in W^+$, it holds that $X \neq (\varepsilon, \varepsilon)$, i.e.,
if there is no solution to $P$, there is no solution to $W$.
\end{lemma}
\begin{proof}
Let $X = X_1 X_2 \cdots X_k$ be the decomposition by parts of $X$.
Assume $X = (\varepsilon, \varepsilon)$ is a solution to $W$, then
since $P$ has no solution by our assumption, Lemma~\ref{lem_concat} proves that $X$ is a concatenation 
of cycles, each of which begins with the same element.
Note further that if any concatenation of cycles $c_{h_1} \cdots c_{h_l}$ 
(where each cycle starts with the same element) equals $(\varepsilon, \varepsilon)$, 
then this implies that we may cyclically permute the product so that it begins with 
element $w_0$ (at least one $w_0$ element must be present in any product of $W$
giving an identity pair since we require at least one cycle).

Due to the `border constraints', Lemma~\ref{lem_concat} gives us 
a restricted form of sequences that may lead to an identity pair, i.e.,
a type $A$ pair of words must be followed by a type $B$ pair of words 
which must be followed by a type $C$ pair of words etc. This implies that
at least one (cyclic) permutation of $X$ must be of the form $ABCD \cdot ABCD \cdots ABCD$ if it equals
$(\varepsilon, \varepsilon)$ since a single cycle is not a solution
to $W$ by Lemma~\ref{oneCycle}.

Assuming that there is no solution to the Restricted PCP instance $P$, for any
part, $Y_i$, we proved in Lemma~\ref{oneCycle} that $Y_i \neq
(\varepsilon, \varepsilon)$, i.e., at least one word in the pairs of
words of each part does not equal $\varepsilon$ (even excluding
initial and final border letters). Thus, crucially, if there exists
no solution to $P$, then $4 \leq \zeta(ABCD) \leq 8$ for any cycle
$ABCD \in W^*$.

We have that $\zeta(Q_1) \geq 4$ for any cycle $Q_1 \in W^*$. 
We shall now prove that $\zeta(Q_1Q_2) \geq 4$ where $Q_2 \in W^*$
is also a cycle, i.e., by adding another cycle to the
existing one, the number of `empty parts' does not decrease. This
means that we cannot reduce such a product to $(\varepsilon,
\varepsilon)$ and thus if there exists no solution to instance $P$,
there exists no solution to the Identity Correspondence Problem
instance $W$ as required. To see this, consider how many parts can
be cancelled by adding a cycle. For example if the first word of $Q_1$ has an
$A$ part which cancels with the $A$ part of $Q_2$, then
the first word for the $B,C,D$ parts of $Q_1$ must be
$\varepsilon$. But since no part can be equal to $(\varepsilon,
\varepsilon)$ we know that in $Q_1$, the second
word of the $B,C,D$ parts must not equal $\varepsilon$. The only
element that can cancel the second word of $Q_1$ is thus the $D$
part of $Q_2$. However this implies that the second word of the 
$A,B,C$ parts of $Q_2$ all equal $\varepsilon$, thus the first word of 
the $B,C$ parts of $Q_2$ cannot be $\varepsilon$ and we have at least four 
non-$\varepsilon$ parts (the first and second words of the $B,C$ parts).

The same argument holds to cancel any part, thus we cannot reduce
more than $4$ parts by the concatenation of any two cycles. The first
word can cancel at most two parts and the second words can cancel at
most two parts but since we start with eight nonempty parts we
remove only four parts at most leaving four remaining parts. Thus
$\zeta(Q_1Q_2) \geq \zeta(Q_1) + \zeta(Q_2) - 4 \geq 4$ as
required. In fact, it is not difficult to see that this argument
can be applied iteratively and thus $\zeta(Q_1Q_2 \cdots Q_m) \geq 4$ always
holds for any $m \geq 1$. If there is no solution to the Restricted PCP 
instance $P$ then a concatenation of cycles cannot form a solution. 
\end{proof}
As an example of this lemma, take the following decomposition by parts (ignoring
`border letters') where $*_i$ is any nonempty word from $\fgr(\Gamma_i)$ (where each $*_i$ 
is understood to be distinct):
$$ ABCD \cdot ABCD =
\left( \frac{\varepsilon}{*_1} \frac{*_2}{\varepsilon} \frac{*_3}{\varepsilon} \frac{*_4}{\varepsilon} \right)
\left( \frac{\varepsilon}{*_1} \frac{\varepsilon}{*_2} \frac{\varepsilon}{*_3} \frac{*_4}{\varepsilon} \right) =
\left( \frac{\varepsilon}{\varepsilon} \frac{*_2}{*_2} \frac{*_3}{*_3} \frac{\varepsilon}{\varepsilon} \right).
$$
Here we cancel four parts in total and we are left with another four
parts. The next $ABCD$ cycle that we concatenate cannot have
$(\varepsilon, \varepsilon)$ for its first two parts however which
will not thus cancel with the last non $\varepsilon$ part above and
thus the next concatenation of $ABCD$ cannot reduce the number of
empty parts by less than four as we showed above, this is the
iterative argument that we apply.


\begin{theorem}
The Identity Correspondence Problem is undecidable for $m = 8(n-1)$ where
$n$ is the minimal number of pairs for which Restricted PCP is known to be
undecidable (currently $n = 7$).
\end{theorem}

\begin{proof}
Given an instance of the Identity Correspondence Problem, $W
\subseteq \fgr(\Gamma) \times \fgr(\Gamma)$ which encodes an
instance of Restricted Post's Correspondence Problem $P$. If there
exists a solution to $P$, Lemma~\ref{PCPSol} shows that there also
exists a solution to $W$. Lemma~\ref{arbProduct} then shows that if there
does not exist a solution to the Restricted PCP instance $P$, there does not 
exist a solution to the Identity Correspondence Problem instance either,
thus proving its undecidability. Since the restricted version of Post's
Correspondence Problem is known to be undecidable for instances of size $7$ by
Theorem~\ref{GPCP}, this implies that ICP is undecidable for $m= 48$ by the construction of $W$.

It remains to prove that we may define the problem over a binary
group alphabet $\{a,b,a^{-1},b^{-1}\}$. This is not difficult
however by a technique which we now outline. Given a group
alphabet $\Sigma_1 = \{y_1, \ldots, y_k, y_1^{-1}, \ldots,
y_k^{-1}\}$ and a binary group alphabet $\Sigma_2 =
\{a,b,a^{-1},b^{-1}\}$. Define $\sigma: \Sigma_1 \to \Sigma_2^*$ by
$\sigma(y_i) = a^iba^{-i}$ and $\sigma(y_i^{-1}) = a^ib^{-1}a^{-i}$ \footnote{Please note 
that this amends a typo in the printed journal article of this paper \cite{BP10}}. It
can be seen that this is an injective morphism (see \cite{BM07}, for example) and
applying iteratively to each letter in each word of $W$ proves the
undecidability of the Identity Correspondence Problem over a binary
group alphabet. 
\end{proof}

\section{Applications of ICP}\label{app_sec}
In this section we will provide a number of new results in 
semigroups using the undecidability of ICP. We first consider the 
``Group Problem'' defined on a semigroup of pairs of words.

\begin{prob} Group Problem - Given an alphabet $\Sigma = \{a,b\}$, is the 
semigroup generated by a finite set of pairs of words 
$P = \{(u_1, v_1), (u_2, v_2), \ldots, (u_m, v_m)\} \subset \fgr(\Sigma)
\times \fgr(\Sigma)$ a group?
\end{prob}

\begin{theorem}
The {\sl Group Problem} is undecidable for $m = 8(n-1)$ pairs of words where
$n$ is the minimal number of pairs for which Restricted PCP is known to be
undecidable (currently $n = 7$).
\end{theorem}
\begin{proof}
Let us assume by contradiction that the Group Problem is decidable
for a semigroup $S$ defined by pairs of words over a group alphabet
and the operation of pairwise concatenation. If the identity element
can be generated by the concatenation of word pairs
$$(u_{i_1},v _{i_1}) (u_{i_2},v _{i_2})\cdot \ldots \cdot (u_{i_k},v _{i_k}) =
(u_{i_1} u_{i_2}\cdot  \ldots \cdot u_{i_k}, v_{i_1} v_{i_2} \cdot
\ldots \cdot v_{i_k}) =(\varepsilon , \varepsilon)$$ then any cyclic
permutation of words in this concatenation is also equal to
$(\varepsilon , \varepsilon)$. Thus every element in the set of all
pairs used in the generation of identity has an inverse element and
this set generates a subgroup. Therefore the Identity Problem can be
solved by checking if some nonempty subset of the original pairs generates
a group. If there is a subset of $S$ which generates a group then
the identity element is in $S$. Otherwise the identity element is
not generated by $S$. 
\end{proof}

It was not previously known whether the Identity Problem for matrix semigroups was
decidable for any dimension greater than two. The Identity Problem in
the two dimensional case for integral matrices was recently proved to be decidable in
\cite{CK05}.

\begin{theorem}\label{IDMatrix}
Given a semigroup $S$ generated by a fixed number $n$ of square four
dimensional integral matrices, determining whether the identity
matrix belongs to $S$ is undecidable. This holds even for $n = 48$.
\end{theorem}
\begin{proof}
We shall use a standard encoding to embed an instance of the
Identity Correspondence Problem into a set of integral matrices.
Given an instance of ICP say $W \subseteq \Sigma^* \times \Sigma^*$
where $\Sigma = \{a,b,a^{-1},b^{-1}\}$ generates a free group. 
Define the morphism $\rho:
\Sigma^* \to \mathbb{Z}^{2 \times 2}$:
$$
\rho(a) = \left(\begin{array}{cc} 1 & 2 \\ 0 & 1 \end{array}\right),
\rho(b) = \left(\begin{array}{cc} 1 & 0 \\ 2 & 1 \end{array}\right),
\rho(a^{-1}) = \left(\begin{array}{cc} 1 & -2 \\ 0 & 1 \end{array}\right),
\rho(b^{-1}) = \left(\begin{array}{cc} 1 & 0 \\ -2 & 1 \end{array}\right).
$$

It is known from the literature that $\rho$ is an injective
homomorphism, i.e., the group generated by $\{\rho(a), \rho(b)\}$ is free, see for example
\cite{LS77}. For each pair of words $(w_1,w_2) \in W$, define the
matrix $A_{w_1,w_2} = \rho(w_1) \oplus \rho(w_2)$ where $\oplus$
denotes the direct sum of two matrices. Let $S$ be a semigroup
generated by $\{A_{w_1,w_2} | (w_1, w_2) \in W\}$. If there exists a
solution to ICP, i.e., $(\varepsilon, \varepsilon) \in W^+$, then we
see that $\rho(\varepsilon) \oplus \rho(\varepsilon) = I_4 \in S$
where $I_4$ is the $4 \times 4$ identity matrix. Otherwise, since
$\rho$ is an injective homomorphism, $I_4 \not\in S$. 
\end{proof}

It follows from the above construction that another open problem
concerning the reachability of any diagonal matrix in a finitely
generated integral matrix semigroup stated in \cite{BCK04} and as
Open Problem~6 in \cite{Har09}, is also undecidable.

\begin{corollary}
Given a finitely generated semigroup of integer matrices $S$, determining whether there
exists any diagonal matrix in $S$ is algorithmically undecidable.
\end{corollary}
\begin{proof}
This result follows from the proof of Theorem~\ref{IDMatrix}. Note that in that
theorem, the morphism $\rho$ is injective and thus the only
diagonal matrix in the range of $\rho$ is the $2 \times 2$
identity matrix  $I_2$ (corresponding to $\rho(\varepsilon)$), 
since diagonal matrices commute. Clearly then, the only
diagonal matrix in the semigroup $S$ of Theorem~\ref{IDMatrix} is
given by $\rho(\varepsilon) \oplus \rho(\varepsilon) = I_4$ where
$I_4$ is the $4 \times 4$ identity matrix. Since determining if this
matrix is in $S$ was shown to be undecidable, it is also undecidable
to determine if there exists any diagonal matrix in $S$. 
\end{proof}


\begin{theorem}
Given a finite set of rotations on the $3$-sphere. Determining
whether this set of rotations generates a group is undecidable.
\end{theorem}
\begin{proof}
We shall use the notation $\mathbb{H}$ to denote the set of quaternions. More details of quaternions used in 
this theorem can be found in \cite{BP08quat}. The set of all unit quaternions forms the unit
$3$-sphere and any pair of unit quaternions $a$ and $b$ can
represent a rotation in $4$ dimensional space. A point
$x=(x_1,x_2,x_3,x_4)$ on the $3$-sphere may be represented by a quaternion
$q_x=x_1+x_2i+x_3j+x_4k$ and rotated using the operation: $a q_x b^{-1}$.
This gives a quaternion $q'_x=x'_1+x'_2i+x'_3j+x'_4k$ representing the rotated point
$x'=(x'_1,x'_2,x'_3,x'_4)$.

We can define a morphism $\xi$ from a group alphabet to unitary quaternions:
$$ \xi(a) = \frac{3}{5}  + \frac{4}{5} \cdot i; \xi(b) = \frac{3}{5}  + \frac{4}{5} \cdot j.$$

It was proven in \cite{BP08quat} that $\xi$ is an injective
homomorphism. We may thus convert pairs of words from an instance of the
Identity Correspondence Problem into pairs of quaternions
$\{(a_1,b_1),\ldots ,(a_n,b_n)\} \subseteq \mathbb{H} \times \mathbb{H}$. Therefore we reduce the Group Problem
for pairs of words over a group alphabet to the question of whether a
finite set of rotations, $\{(a_1,b_1),\ldots ,(a_n,b_n)\}$,
represented by pairs of quaternions, generates a group. 
\end{proof}

\section{Conclusion}

In this paper we introduced the Identity Correspondence Problem,
proved that it is undecidable and applied it to answer  long
standing open problems in matrix semigroups. In particular, we proved that the membership 
problem for the identity matrix in $4 \times 4$ integral matrix semigroups is undecidable. The identity matrix membership
problem for $2 \times 2$ matrix semigroups was shown to be decidable in \cite{CK05}, but the problem in dimension $3$ remains open.
We believe that the Identity Correspondence Problem will be useful in identifying new
areas of undecidable problems not only related to matrix problems
but also to computational questions in abstract algebra, logic and
combinatorics on words.



\vskip 0.4cm

\noindent{\bf Acknowledgements} - We would like to thank Prof. Tero
Harju for useful discussions concerning this problem and the anonymous referees for their careful checking of this manuscript.


\end{document}